\documentclass[11pt]{article}
\usepackage{color}

\usepackage{amssymb}   
\usepackage{amsthm}    
\usepackage{amsmath}   
\usepackage{stmaryrd}  
\usepackage{titletoc}  
\usepackage{mathrsfs}  
\usepackage{graphicx}

\newlength{\defbaselineskip}
\newcommand{\setlinespacing}[1]%
           {\setlength{\baselineskip}{#1 \defbaselineskip}}

\theoremstyle{plain}
\newtheorem{thm}{Theorem}[section]
\newtheorem{cor}[thm]{Corollary}
\newtheorem{lem}[thm]{Lemma}
\newtheorem{prop}[thm]{Proposition}

\theoremstyle{definition}
\newtheorem{defn}{Definition}[section]

\newtheorem{rmk}{Remark}[section]

\newcommand{\eps}{\varepsilon}

\DeclareMathOperator*{\esssup}{esssup}

\newcommand{\cH}{\mathcal{H}}
\newcommand{\cL}{\mathcal{L}}

\newcommand{\cA}{\mathcal{A}}
\newcommand{\cS}{\mathcal{S}}

\newcommand{\bP}{\mathbb{P}}
\newcommand{\bR}{\mathbb{R}}
\newcommand{\bN}{\mathbb{N}}

\newcommand{\bV}{\mathbb{V}}
\newcommand{\bL}{\mathbb{L}}
\newcommand{\sF}{\mathscr{F}}

\makeatletter\@addtoreset{equation}{section} \makeatother
 \allowdisplaybreaks

\begin{document}

\title{ $L^2$-Theory of Linear Degenerate SPDEs and $L^p$ ($p>0$) Estimates for the Uniform Norm of Weak Solutions
}

\author{Jinniao Qiu\footnotemark[1]  }
\footnotetext[1]{Department of Mathematics \& Statistics, University of Calgary, 2500 University Drive NW, Calgary, AB T2N 1N4, Canada. \textit{E-mail}: \texttt{jinniao.qiu@ucalgary.ca}. 

The author is partially supported by the National Science and Engineering Research Council of Canada (NSERC) and by the start-up funds from the University of Calgary.}

%
%

\maketitle

\begin{abstract}
In this paper, we are concerned with possibly degenerate stochastic partial differential equations (SPDEs).  An $L^2$-theory is introduced, from which we derive a H\"ormander-type theorem with an analytical approach. With the method of De Giorgi iteration, we obtain  the maximum principle which states the $L^p$ ($p>0$) estimates for the time-space uniform norm of weak solutions.

\end{abstract}




\section{Introduction}

Let $(\Omega,{\sF},\{{\sF}_t\}_{t \geq 0},\bP)$ be a complete filtered probability space, on which a $d_1$-dimensional Wiener process $W=(W_t)_{t\geq 0}$ is well defined.  
We consider SPDE of the form
\begin{equation}\label{SPDE}
  \left\{\begin{array}{l}
  \begin{split}
  du(t,x)=\,&\displaystyle \left[ \frac{1}{2}(L_k^2+M_k^2)u
         +{b}^jD_ju+cu+f+L_k'g^k+M_k'h^k
                \right](t,x)\, dt\\ &\displaystyle
           +\left[M_ku+\beta^k u+h^k\right](t,x)\, dW_{t}^{k}, \quad
                     (t,x)\in Q:= [0,T]\times \bR^d;\\
    u(0,x)=\, &\underline{u}_0(x), \quad x\in\bR^d.
    \end{split}
  \end{array}\right.
\end{equation}
Here and throughout this paper,  the summation over repeated indices is enforced unless stated otherwise,  $T\in (0,\infty)$, $D=(D_1,\dots,D_d)$ is the gradient operator, and $L_k=\sigma^{jk}D_j$, $M_k=\theta^{jk}D_j$, $L_k'=D_j(\sigma^{jk}\cdot)$, $M_k'=D_j(\theta^{jk}\cdot)$, for $k=1,\dots,d_1$. SPDE \eqref{SPDE} is said to be degenerate when it fails to satisfy the super-parabolicity (\textbf{SP}): There exists   $\lambda\in(0,\infty)$  such that
   \begin{align*}
       \sigma^{ik}\sigma^{jk}(t,x)\xi^i\xi^j\geq \lambda |\xi|^2\,\,\,\,\,a.s.,\,\,\,\,\,\forall\, (t,x,\xi)\in [0,T]\times\bR^d\times\bR^d.
   \end{align*}

We first investigate the solvability of linear, possibly degenerate SPDEs in $L^2$-spaces. An $L^2$-theory on linear degenerate SPDEs was initiated by Krylov and Rozovskii \cite{krylov_1986charcteristic,krylov-1977-Deg-L2}, and it was developed recently by \cite{Deg-SPDE-3-2014,Deg-SPDE-2014,krylov2013hypoellipticity,Deg-SPDE-2-2014}. Along this line, obtaining a solution of SPDE \eqref{SPDE} in space $L^2(\Omega;C([0,T];H^m))$ not only requires that $f+L_k'g^k+M_k'h^k$ is $H^m$-valued but also assumes that $h^k$ is $H^{m+1}$-valued, while in this work, $f,g$ and $h$ are allowed to be just $H^m$-valued.
Moreover, we get the estimate $L_ku\in L^2(\Omega\times[0,T];H^m) $, and under a H\"ormander-type condition, we further have $u\in L^2(\Omega\times[0,T];H^{m+\eta})$ for some $\eta\in(0,1]$. For the proof, we apply the a priori estimates for solutions of the  approximating \textit{super-parabolic} SPDEs in line with the applications of pseudo-differential operator theory. As a byproduct, a H\"ormander-type theorem for SPDE \eqref{SPDE} is derived from the established $L^2$-theory and an estimate on the Lie bracket (Lemma \ref{lem-lie-bracket}).

Most importantly, we prove the maximum principle for the weak solution of SPDE \eqref{SPDE}. More precisely, we obtain the $L^p$ ($p>0$) estimates for the time-space uniform norm of weak solutions, i.e., under suitable integrability assumptions on $\underline{u}_0,f,g$ and $h$, we have
\begin{thm}\label{thm-MP-intr}
Let the H\"ormander-type condition $({\mathcal H} )$ hold. For the weak solution $u$ of SPDE \eqref{SPDE}, we have for any $p\in (0,\infty)$
\begin{align*}
E\|u^{\mp}\|^p_{L^{\infty}(Q)} \leq C\,\Xi (\underline{u}_0^{\mp},f^{\mp},g,h),
\end{align*}
where $\Xi (\underline{u}_0^{\mp},f^{\mp},g,h)$ is expressed in terms of certain norms of $(\underline{u}_0,f^{\mp},g,h)$, and  the constant $C$ depends on $d,p,T$ and the quantities related to the structure coefficients of SPDE \eqref{SPDE}.
\end{thm}

The novelty of our result is that it does not require the super-parabolic condition (\textbf{SP}), which, to the best of our knowledge, is always assumed in the existing literature on such kind of maximum principles for SPDEs. 

For the super-parabolic SPDEs, Krylov \cite{Kryl96} established the $L^p$-theory ($p\geq 2$), from which one can derive from the classical Sobolev embedding theorem the $L^p$ estimates of time-space uniform norm for the \textit{strong} solutions that require stronger smoothness assumptions on the coefficients. For the weak solutions of super-parabolic SPDEs in bounded domains, the maximum principle was obtained by Denis, Matoussi and Stoica \cite{DenisMatoussiStoica2005} and further by \cite{D-bounddeness-SPDE-2013,DenisMatoussi2009}, but with $p\in [2,\infty)$. Their method relied on Moser's iteration. Such method was also used by Denis, Matoussi and Zhang \cite{denis-2014-MP} to derive the maximum principle for weak solutions of super-parabolic SPDEs with obstacle. In comparison, we adopt a stochastic version of De Giorgi iteration scheme in this paper. We would also note that our method is inspired by the other two different versions of De Giorgi iteration used by Hsu, Wang and Wang \cite{hsu-2013-Stoch-DeGiorgi} to investigate the regularity of \textit{strong} solutions for \textit{super-parabolic} SPDEs and by Qiu and Tang \cite{QiuTangMPBSPDE11} to study the maximum principles of weak solutions for quasilinear \textit{backward} SPDEs. For some more recent works on supremum bounds for solutions of SPDEs with iteration methods, we refer to \cite{dareiotis2017local,dareiotis2017supremum,G-2019,wang2018probabilistic}; especially, \cite{dareiotis2017supremum} deals with some classes of degenerate nonlinear SPDEs under \textit{strong} uniform boundedness assumptions on the external force terms.

The remainder of this paper is organized as follows. In Section 2, we set some notations and state our main result. The $L^2$-theory and the H\"ormander-type theorem are addressed in Section 3. Finally, we prove the maximum principle in section 4.

\section{Preliminaries and the main results}


Let $L^2(\bR^d)$ ($L^2$ for short) be the usual Lebesgue integrable space with usual scalar product $\langle\cdot,\,\cdot\rangle$ and norm $\|\cdot\|$.
For  $n\in
(-\infty,\infty)$, we denote by $H ^{n}$ the space of Bessel
potentials, that is
$H ^{n}:=(1-\Delta)^{-\frac{n}{2}} L^{2}$
with the norm
$$\|\phi\|_{n}:=\|(1-\Delta)^{\frac{n}{2}}\phi\|, ~~ \phi\in
H^{n}.$$  
For each $l\in \mathbb{N}^+$ and domain $\Pi\subset \bR^l$, denote by $C_c^{\infty}(\Pi)$ the space of infinitely differentiable functions with compact supports in $\Pi$. For convenience, we shall use $\langle \cdot,\,\cdot\rangle$ to denote the duality between $(H^n)^k$ and $(H^{-n})^k$ ($k\in\bN^+,\,n\in\bR$) as well as that between the Schwartz function space $\mathscr{D}$ and  $C_c^{\infty}(\bR^d)$. Moreover, we always omit the superscript associated to the dimension when there is no confusion.

For  Banach space ($\mathbb{B}$, $\|\cdot\|_{\mathbb{B}}$) and $p\in[1,\infty]$, $\cS ^p (\mathbb{B})$ is the set of all the $\mathbb{B}$-valued,
 $(\sF_t)$-adapted and continuous processes $(X_{t})_{t\in [0,T]}$ such
 that
 $$\|X\|_{\cS ^p(\mathbb{B})}:= \left\| \sup_{t\in [0,T]} \|X_t\|_{\mathbb{B}} \right\|_{L^p(\Omega)}< \infty.$$
  Denote by $\mathcal{L}^p(\mathbb{B})$ the totality of all the $\mathbb B$-valued,
 $(\sF_t)$-adapted processes $(X_{t})_{t\in [0,T]}$ such
 that
 $$
 \|X\|_{\mathcal{L}^p(\mathbb{B})}:= \left\| \|X_t\|_{\mathbb{B}} \right\|_{L^p(\Omega\times[0,T])}< \infty.
 $$
Obviously, both $(\cS^p(\mathbb{B}),\,\|\cdot\|_{\cS^p(\mathbb{B})})$ and  $(\mathcal{L}^p(\mathbb{B}),\|\cdot\|_{\mathcal{L}^p(\mathbb{B})})$ 
 are Banach spaces. In addition, for $p\in(0,1)$, we denote by $L^p(\Omega;\mathbb B)$ the $\mathbb B$-valued $\sF$-measurable functions $f$ such that $\|f\|_{\mathbb B}^p\in L^1(\Omega;\bR)$ with 
 $\|f\|_{L^p(\Omega;\mathbb B)}:=\left\|\|f\|_{\mathbb B}^p\right\|_{L^1(\Omega;\bR)}^{1/p}$.

%
By $C_b^{\infty}$, we denote the set of infinitely differentiable functions with bounded derivatives of any order.  Denote by $\cL^{\infty}(C_b^{\infty})$ the set of functions $h$ on $\Omega\times [0,T] \times \bR^d$ such that $h(t,x)$ is infinitely differentiable with respect to $x$ and all the derivatives of any order belong to $ \cL^{\infty}(L^\infty(\bR^d))$.

Throughout this paper, we denote $I^n=(1-\Delta)^{\frac{n}{2}}$ for $n\in\bR$. Then $I^n$ belongs to $\Psi_n$ that is the class of pseudo-differential operators of order $n$. By the pseudo-differential operator theory (see \cite{Hormander1983analysis} for instance), the $m$-th order differential operator belongs to $\Psi_m$ for $m\in\bN^+$, the multiplication by elements of $C_b^{\infty}$ lies in $\Psi_0$, and for the reader's convenience, two basic results are collected below.
\begin{lem}\label{lem-pdo}
(i). If $J_1\in\Psi_{n_1}$ and $J_2\in\Psi_{n_2}$ with $n_1,n_2\in\bR$, then $J_1J_2\in\Psi_{n_1+n_2}$ and the Lie bracket $[J_1,J_2]:=J_1J_2-J_2J_1\in\Psi_{n_1+n_2-1}$.

(ii). For $m\in (0,\infty)$, let $\zeta$ belong to $C_b^{m}$ which is defined as usual. Then for any $n\in(-m,m)$ there exists constant $C$ such that 
$$
\|\zeta\phi\|_n\leq C \|\zeta\|_{C^{m}}\|\phi\|_n,\quad \forall\,\phi\in H^n.
$$ 
\end{lem}

We introduce the definition for solution of SPDE \eqref{SPDE}.
\begin{defn}\label{defn-solution}
  A process $u$
  is called a solution to SPDE \eqref{SPDE}
  if $u\in \cS^2(H^m)$ for some $m\in\bR$ and SPDE   \eqref{SPDE} holds in the distributional sense, i.e.,
   for any $\zeta\in C^{\infty}_c(\bR)\otimes C_c^{\infty}(\bR^d)$ there holds almost surely
  \begin{equation*}
    \begin{split}
      &\langle \zeta(t),\,u(t)\rangle
      -\!\int_0^t\!\!\langle \partial_s \zeta(s),\, u(s) \rangle \, ds
      -\!\int_0^t \!\!\langle \zeta(s),\, (M_ku+\beta^ku+h^k)(s)\rangle  \,dW_s^k
      \\
      &=\langle \zeta(0),\, \underline u_0\rangle +
      \int_0^t\!\! \bigg\langle  \zeta,\,
      \frac{1}{2}(L_k^2+M_k^2)u
         +{b}^jD_ju+cu+f+L_k'g^k+M_k'h^k
          \bigg\rangle(s)\, ds , \quad \forall \ t\in[0,T].
    \end{split}
  \end{equation*}
  In particular, if $u\in \cS^2(L^2)$, it is said to be a weak solution.
\end{defn}

Set 
$$\bV_0=\{L_1,\dots,L_{d_1}\} \quad \text{and} \quad \bV_{n+1}=\bV_n\cup\{[L_k,V]:\,V\in\bV_n,\,k=1,\dots,d_1\}.$$
Denote by $\bL_n$ the set of linear combinations of elements of $\bV_n$ with coefficients of $\cL^{\infty}(C^{\infty}_b)$. We introduce the following H\"ormander-type condition.

\medskip
$({\mathcal H} )$ \it
   There exists $n_0\in\bN_0$ such that $\{D_i:i=1,\dots,d\}\subset \bL_{n_0}$. (Throughout this paper, $n_0$ is always chosen to be the smallest one.) \rm
   \medskip

\begin{rmk}
It is obvious that the super-parabolicity (\textbf{SP}) corresponds to the trivial case $n_0=0$. A nontrivial example is the $2$-dimensional case with $d_1=d=2$: $L_1=D_1$ and $L_2=\cos \left((1+\alpha_t)x_1\right) D_2$ where $(\alpha_t)_{t\geq 0}$ can be any nonnegative bounded $\sF_t$-adapted process. Then one has $D_2\notin \bL_0$, but $\{D_1,D_2\}\subset \bL_1$ since $[L_1,L_2]=-(1+\alpha_t)\sin \left((1+\alpha_t)x_1\right) D_2$. Hence, we have $n_0=1$.
\end{rmk}

We also make the following assumptions.
\medskip\\
   $({\mathcal A} 1)$ \it 
   $\sigma^{ik},\theta^{ik},b^i,\beta, c\in\cL^{\infty}(C_b^{\infty})$, for $i=1,\dots,d$, $k=1,\dots,d_1$;
   \medskip\\
   $({\mathcal A} 2)$ \it    $c\geq 0$,   $\underline{u}_0\in L^{\infty}(\Omega\times\bR^d)\cap \cap_{q>0} L^q(\Omega,\sF_0;L^2)$,  $f,g^k,h^k\in \cL^2(L^2)\cap \cap_{q>0} L^q(\Omega;L^2(Q))$, for $k=1,\dots,d_1$, and moreover, for some $ \bar p >d+2\eta$
   $$(f,g,h)\in L^{\infty}(\Omega;L^{\frac{\bar p (d+2\eta)}{(\bar p + d + 2\eta)\eta}}(Q))\times 
   L^{\infty}(\Omega;L^{\frac{\bar p }{\eta}}(Q)) \times
    \left(
    L^{\infty}(\Omega;L^{\frac{2\bar p (d+2\eta)}{(\bar p + d + 2\eta)\eta}}(Q))  
    \cap L^{\infty}(\Omega;L^{\frac{\bar p }{\eta}}(Q))\right), $$
where and in the following,  we set $\eta=2^{-n_0}$. Throughout this paper, we denote
\begin{align*}
\Lambda_{\bar p,\infty}^{\mp}&= \|\underline{u}_0^{\mp}\|_{L^{\infty}(\Omega\times\bR^d)} +\esssup_{\omega\in\Omega}\|f^{\mp}(\omega,\cdot,\cdot)\|_{L^{\frac{\bar p (d+2\eta)}{(\bar p + d + 2\eta)\eta}}(Q)}
+\esssup_{\omega\in\Omega}\|(g,h)(\omega,\cdot,\cdot)\|_{L^{\frac{\bar p }{\eta}}(Q)}\\
&\quad
+\esssup_{\omega\in\Omega}\|h(\omega,\cdot,\cdot)\|_{L^{\frac{2\bar p (d+2\eta)}{(\bar p + d + 2\eta)\eta}}(Q)}
,\\
\Lambda_p^{\mp}&=\|\underline{u}_0^{\mp}\|_{L^p(\Omega;L^2)} + \|(f^{\mp},g,h)\|_{L^p(\Omega;L^2(Q))} ,\quad p\in(0,\infty).
\end{align*}
\rm

We now state our main results. 
\begin{thm}\label{thm-main}
Let assumption $(\cA 1)$ hold. Given $f\in \cL^2(H^m)$, $g,h\in \cL^2((H^m)^{d_1})$ and $\underline{u}_0\in L^2(\Omega,\sF_0;H^m)$ with some $m\in\bR$, the following three assertions hold:

(i) SPDE \eqref{SPDE} admits a unique solution $u\in\cS^2(H^m)$ with $L_ku\in\cL^2(H^m)$, $k=1,\dots,d_1$, and
\begin{align*}
&E\sup_{t\in[0,T]}\|u(t)\|_m^2+\sum_{k=1}^{d_1}E\int_{0}^T\|L_ku(t)\|_m^2\,dt
\nonumber\\
&\leq C\left\{
E\|\underline{u}_0\|_m^2 + E\int_{0}^T \left(\|f(s)\|_m^2+\|g(s)\|_m^2+\|h(s)\|_m^2\right)\,ds\right\},
\end{align*}
with the constant $C$ depending on $T,m,\theta,\sigma,b,c$ and $\beta$. In particular, if condition $(\cH)$ holds, we have further
\begin{align*}
E\int_{0}^T\|u(t)\|_{m+\eta}^2\,dt
\leq C\left\{
E\|\underline{u}_0\|_m^2 + E\int_{0}^T \left(\|f(s)\|_m^2+\|g(s)\|_m^2+\|h(s)\|_m^2\right)\,ds\right\},
\end{align*}
with $C$ depending on $T,m,n_0,\theta,\sigma,b,c$ and $\beta$.

(ii) Assume further $({\mathcal H} )$ and $f\in\cap_{n\in\bR}\cL^2(H^n)$, $g,h\in\cap_{n\in\bR}\cL^2((H^n)^{d_1})$. For any $\eps\in(0,T)$, one has $u\in\cap_{n\in\bR} L^2(\Omega;C([\eps,T];H^n))$, and for each $n\in\bR$,
\begin{align}
&E\sup_{t\in[\eps,T]}\| u(t)\|_{n}^2
+E\int_{\eps}^{T} \| u(t)\|_{n+\eta}^2\,dt
\nonumber\\
&\leq C\left\{E\|\underline{u}_0\|_m^2+
E\int_{0}^{T} \left(\|f(s)\|_{n}^2+ \|g(s)\|_{n}^2+ \|h(s)\|_{n}^2\right) ds\right\},\label{est-hmder}
\end{align}
with the constant $C$ depending on $\eps,n,T,m,n_0,\sigma,\theta,\gamma,b$ and $c$. In particular, the random field $u(t,x)$ is almost surely infinitely differentiable with respect to $x$ on $(0,T]\times\bR^d$ and each derivative is a continuous function on $(0,T]\times\bR^d$.

(iii) Let assumption $(\cA 2)$ and condition $(\cH)$ hold. For the weak solution $u$ of SPDE \eqref{SPDE}, there exists $\theta_0\in(0,1]$ such that for any $p>0$,
\begin{align*}
E\|u^{\mp}\|^p_{L^{\infty}(Q)} \leq C \left(\Lambda^{\mp}_{\bar p,\infty} + \Lambda^{\mp}_{\frac{p}{\theta_0}}  \right)^p,
\end{align*}
with the constant $C$ depending on $d,p,n_0, T$ and the quantities related to the coefficients $\sigma,\theta,b,c$ and $\beta$.
\end{thm}

\begin{rmk}\label{rmk-mainr}
 Assertion (i) is a summary of Theorem \ref{thm-BSPDE} and Corollary \ref{cor-grad-est}, in which an $L^2$-theory is presented for the linear, possibly degenerate SPDEs. 
Assertion (ii) is from Theorem \ref{thm-hormander}, which is a H\"ormander-type theorem. The most important result of this paper is the maximum principle of assertion (iii), which corresponds to Theorem \ref{thm-Lp-inf} below and states the $L^p$ ($p>0$) estimates for the time-space uniform norm of weak solutions for possibly \textit{degenerate} SPDE \eqref{SPDE} in the whole space. 
\end{rmk}


\section{$L^2$-theory and H\"ormander-type theorem for SPDEs}
\subsection{$L^2$-theory of SPDEs}
\label{sec-l2-thy}

We consider the following SPDE
\begin{equation}\label{SPDE-D}
  \left\{\begin{array}{l}
  \begin{split}
  du(t,x)=\,&\displaystyle \left[\delta\Delta u+ \frac{1}{2}(L_k^2+M_k^2)u
         +{b}^jD_ju+cu+f+L_k'g^k+M_k'h^k
                \right](t,x)\, dt\\ &\displaystyle
           +\left[M_ku+\beta^k u+h^k\right](t,x)\, dW_{t}^{k}, \quad
                     (t,x)\in Q;\\
    u(0,x)=\, &\underline{u}_0(x), \quad x\in\bR^d,
    \end{split}
  \end{array}\right.
\end{equation}
with $\delta \in [0,\infty)$.

We first give an a priori estimate for the solution of SPDE  \eqref{SPDE-D}.
\begin{prop}\label{prop-apriori-estm}
Let assumption $(\cA 1)$ hold. Assume $\underline{u}_0\in L^2(\Omega,\sF_0;H^m)$ and $f,g^k,h^k\in \cL^2(H^m)$ with $m\in \bR$, for $k=1,\dots,d_1$.  If $u\in \cS^2(H^{m+1})\cap\cL^2(H^{m+2})$ is a solution of SPDE \eqref{SPDE-D}, one has
\begin{align}
&E\sup_{t\in[0,T]}\|u(t)\|_m^2+E\int_{0}^T\left(\delta \|Du(t)\|_m^2+\sum_{k=1}^{d_1}\|L_ku(t)\|_m^2\right)\,dt
\nonumber\\
&\leq C\left\{
E\|\underline{u}_0\|_m^2 + E\int_{0}^T \left(\|f(s)\|_m^2+\|g(s)\|_m^2+\|h(s)\|_m^2\right)\,ds\right\},\label{estim-aprioi-prop}
\end{align}
with $C$ being independent of $\delta$.
\end{prop}

\begin{proof}
 We have decompositions $L_k=L_k'+c_k$ and $M_k=M_k'+\alpha_k$ with $c_k=-(D_i\sigma^{ik})\cdot$ and $\alpha_k=-(D_i\theta^{ik})\cdot$, for $k=1,\dots,d_1$.
 Applying It\^o formula for the square norm (see e.g. \cite[Theorem 3.1]{Krylov_Rozovskii81}), one has almost surely for $t\in[0,T]$,
\begin{align}
&\|I^mu(t)\|^2+\int_0^t  2\delta\|I^mDu(s)\|^2\,ds
-\int_0^t \!\!\!2\langle I^m u(s),\,I^m((Du)\theta+\beta u+h)(s) \,dW_s\rangle  \nonumber
\\
&=\|I^m\underline{u}_0\|^2+
\int_0^t\left\langle
I^mu,\,
I^m\left((L_k^2+M_k^2)u+2M_k'h^k+2L_k'g^k
         \right)\right\rangle(s)\,ds
         \nonumber\\
         &\quad
         +\int_0^t2\left\langle I^mu,\,I^m\left(b^jD_ju+cu+f\right)\right\rangle(s)\,ds
         +\int_0^t\|I^m((Du)\theta+\beta u+h)(s)\|^2\,ds. \label{eq-prop-ito}
\end{align}

First, basic calculations yield
\begin{align}
&\langle I^mu,\,I^m(L_k^2u+2L_k'g^k+2cu+2f)\rangle\nonumber\\
&=\langle I^mu,\,I^m(L'_k+c_k)L_ku\rangle+2\langle I^mu,\,I^m L'_kg^k\rangle
+2\langle I^mu,\,I^m(cu+f)\rangle\nonumber\\
&=-\| I^mL_ku\|^2 +\langle [I^m,L_k]u,\,I^mL_ku\rangle+\langle I^mu,\,[I^m,L_k']L_ku+ I^mc_kL_ku\rangle
\nonumber\\
&\,\,\,
-2\langle I^mL_ku,\,I^mg^k\rangle+2\langle[I^m,L_k]u,\,I^m g^k\rangle +2\langle I^mu,\,[I^m,L_k']g^k\rangle
+2\langle I^mu,\,I^m(cu+f)\rangle\nonumber\\
&\leq-(1-\eps)\|I^mL_ku\|^2+C_{\eps}\left(\|I^mu\|^2+\|I^mg^k\|^2+\|I^mf\|^2\right),\quad\eps\in(0,1),\label{eq-rela-u}
\end{align}
and
\begin{align}
&\langle I^mu,\,I^m(M_k^2u+2M_k'h^k)\rangle\nonumber\\
&=-\| I^mM_ku\|^2 +\langle [I^m,M_k]u,\,I^mM_ku\rangle+\langle I^mu,\,[I^m,M_k']M_ku+ I^m\alpha_kM_ku\rangle
\nonumber\\
&\,\,\,
-2\langle I^mM_ku,\,I^mh^k\rangle+2\langle[I^m,M_k]u,\,I^m h^k\rangle +2\langle I^mu,\,[I^m,M_k']h^k\rangle
\nonumber\\
&\leq
-\| I^mM_ku\|^2 -2\langle I^mM_ku,\,I^mh^k\rangle 
+\langle [I^m,M_k]u,\,M_kI^mu\rangle+\langle I^mu,\,[I^m,M_k]M_ku+ \alpha_kM_kI^mu\rangle
\nonumber\\
&\,\,\,
+C \left(\|I^mu\|^2+\|I^mh^k\|^2\right)\nonumber\\
&\leq
-\| I^mM_ku\|^2 -2\langle I^mM_ku,\,I^mh^k\rangle 
+\langle I^mu,\,[[I^m,M_k],M_k]u+ \alpha_kM_kI^mu\rangle
\nonumber\\
&\,\,\,
+C \left(\|I^mu\|^2+\|I^mh^k\|^2\right)\nonumber\\
&\leq
-\| I^mM_ku\|^2 -2\langle I^mM_ku,\,I^mh^k\rangle 
+
C \left(\|I^mu\|^2+\|I^mh^k\|^2\right),\label{eq-relat-uxi}
\end{align}
where we have used the relation 
\begin{align}
\langle I^mu,\,\alpha_kM_kI^mu\rangle=-\frac{1}{2}\langle I^mu,\,D_i(\alpha_k\theta^{ik})I^mu\rangle.\label{eq-relat-adu}
\end{align}

Notice that for $i=1,\dots,d$, $k=1,\dots,d_1$,
\begin{align}
\|I^m(h^k+\beta^k u+M_ku)\|^2
&=\|I^mh^k\|^2+2\langle I^mh^k,\,I^mM_ku\rangle+\|I^mM_ku\|^2\nonumber\\
&\,\,\,\,+2\langle I^m(h^k+M_ku), I^m(\beta^k u)\rangle + \|I^m(\beta^k u)\|^2,\label{esti-diff-prop3-1}\\
\langle I^mu,\,I^m(b^iD_iu)\rangle 	&=-\frac{1}{2}
\langle I^mu,\,D_ib^iI^mu
+2 [b^iD_i,\,I^m]u\rangle,\nonumber\\
\langle I^mM_ku,\,I^m(\beta^ku) \rangle
&\leq
\langle M_kI^mu,\,\beta^kI^mu\rangle+C\|I^m u\|^2\nonumber\\
&=
-\frac{1}{2} \langle I^mu, \,D_i(\beta^k\theta^{ik})I^mu\rangle
+C\|I^m u\|^2.\nonumber
\end{align} 
 Putting \eqref{eq-prop-ito}, \eqref{eq-rela-u} and \eqref{eq-relat-uxi} together, and taking expectations on both sides of \eqref{eq-prop-ito}, one gets by Gronwall inequality
\begin{align}
&\sup_{t\in[0,T]}E\|u(t)\|_m^2+E\int_{0}^T\left(\delta\|Du(t)\|_m^2+\sum_{k=1}^{d_1}\|L_ku(t)\|_m^2\right)dt
\nonumber\\
&\leq C\left\{
E\|\underline{u}_0\|_m^2 + E\int_{0}^T \left(\|f(s)\|_m^2+\|g(s)\|_m^2+\|h(s)\|_m^2\right)\,ds\right\}. \label{eq-prf-0}
\end{align}
On the other hand, one has for each $t\in[0,T)$,
\begin{align*}
&E\sup_{\tau\in[0,t]}\bigg|\int_0^{\tau} \!\!\!2\left\langle I^m u(s),\,I^m(h+\beta u+(Du)\theta)(s)\,dW_s\right\rangle\bigg|\\
&\leq C \bigg(E\sum_{k=1}^{d_1}\int_0^t\left(|\langle I^mu(s),\,I^m(h^k+\beta^ku)(s)\rangle|^2 +|\langle I^mu(s),\,(M_kI^m+[I^m,M_k])u(s)\rangle|^2\right)ds\bigg)^{1/2}\\
&\leq C \bigg(E\int_0^t
\left(\|I^mu(s)\|^2\|I^mh(s)\|^2+\|I^mu(s)\|^4\right)ds
\bigg)^{1/2	}
\\
&\leq \eps E\sup_{s\in[0,t]}\|I^mu(s)\|^2 +C_{\eps} E\int_0^t
\left(\|I^mh(s)\|^2+\|I^mu(s)\|^2\right)ds,\quad \eps\in(0,1).
\end{align*}
Together with \eqref{eq-prop-ito}, \eqref{eq-rela-u}, \eqref{eq-relat-uxi} and \eqref{eq-prf-0}, the above estimate implies \eqref{estim-aprioi-prop}.

\end{proof}

\begin{rmk}
The estimate \eqref{estim-aprioi-prop} plays an important role in our $L^2$-theory for SPDEs, for which some unusual techniques are applied in the calculations of \eqref{eq-rela-u}, \eqref{eq-relat-uxi} and \eqref{esti-diff-prop3-1}. Especially, we treat the term $2\langle I^mM_ku,\,I^mh^k\rangle$ as a unity and it allows us to weaken the assumptions on $h$ in the $L^2$-theory.
\end{rmk}
An immediate consequence of Proposition \ref{prop-apriori-estm} is the following 
\begin{cor}\label{cor-uniqn}
Let assumption $(\cA 1)$ hold. Given $\underline{u}_0\in L^2(\Omega,\sF_0;H^m)$ and $f,g^k,h^k\in \cL^2(H^m)$ with $m\in \bR$, for $k=1,\dots,d_1$, the solution of SPDE \eqref{SPDE-D} is unique.
\end{cor}

\begin{thm}\label{thm-BSPDE}
Let assumption $(\cA 1)$ hold. Assume $\underline{u}_0\in L^2(\Omega,\sF_0;H^m)$ and $f,g^k,h^k\in \cL^2(H^m)$ with $m\in \bR$, for $k=1,\dots,d_1$. SPDE  \eqref{SPDE-D} with $\delta=0$ (equivalently, SPDE \eqref{SPDE}) admits a unique solution $u\in \cS^2(H^{m})$ with $L_ku\in\cL^2(H^m)$, $k=1,\dots,d_1$, and
\begin{align}
&E\sup_{t\in[0,T]}\|u(t)\|_m^2+\sum_{k=1}^{d_1}E\int_{0}^T\|L_ku(t)\|_m^2\,dt
\nonumber\\
&
\leq C\left\{
E\|\underline{u}_0\|_m^2 + E\int_{0}^T \left(\|f(s)\|_m^2+\|g(s)\|_m^2+\|h(s)\|_m^2\right)\,ds\right\},\label{estim-thm}
\end{align}
with $C$ depending on $T,m,\sigma,\theta,b,c$ and $\beta$.
\end{thm}

\begin{proof}
Choose $\{\delta_l\}_{l\in\bN^+}\subset(0,1)$, 
$\{\underline u^n_0\}_{n\in\bN^+}\subset L^2(\Omega,\sF_0;H^{m+5})$ and $\{f_n,g_n^k,h_n^k\}_{n\in\bN^+}\subset \cL^2(H^{m+5})$, for $k=1,\dots,d_1$, such that $\delta_l$ converges down to $0$ and 
$$
\lim_{n\rightarrow \infty} \,\|\underline{u}_0^n-\underline{u}_0\|_{L^2(\Omega;H^m)} 
+\|\left(f_n-f,g_n-g,h_n-h   \right)\|_{\cL^2(H^m)}
\,=0.
$$
By $L^p$-theory for SPDEs (see \cite{Kryl96} for instance), SPDE \eqref{SPDE-D} admits a unique solution $u_{l,n}\in \cS^2(H^{m+5})\cap \cL^2(H^{m+6}) $ associated with $(\delta_l,f_n,g_n,h_n,\underline{u}^n_0)$. 
 
 Fixing $n$, one deduces from Proposition \ref{prop-apriori-estm} that  $\{(u_{l,n},L_ku_{l,n})\}_{l\in\bN^+}$ is bounded in $\cS^2(H^{m+4})\times \cL^2(H^{m+4})$, $k=1,\dots,d_1$. Observe that $\delta_l\Delta u_{l,n}$ tends to zero in $\cL^2(H^{m+2})$ as $l$ goes to infinity. Therefore, letting $l$ tend to infinity, we derive from Proposition \ref{prop-apriori-estm} and Corollary \ref{cor-uniqn} the unique solution $u_n$ for SPDE \eqref{SPDE-D} associated with $(f_n,g_n,h_n,\underline{u}_0^n)$ and $\delta=0$ such that $(u_n,L_ku_n)\in\cS^{2}(H^{m+2})\times \cL^2(H^{m+2})$, for $k=1,\dots,d_1$. 
 
 Furthermore, letting $n$ go to infinity, again by Proposition \ref{prop-apriori-estm} and Corollary \ref{cor-uniqn}, one obtains the unique solution $u$ and associated estimates. This completes the proof.
\end{proof}
Here, we would note that the above proof is based on methods of strong convergence which is different from the weak convergence developed in \cite{krylov_1986charcteristic}. This is basically because of the linearity of the concerned equations and the smoothness assumptions on coefficients in $(\mathcal A 1)$, and it makes the passage to limits more straightforward through approximations.

\begin{rmk}\label{rmk-ito-forml}
Consider the particular case $m=0$ in Theorem \ref{thm-BSPDE}. In view of the approximations in the above proof, through similar calculations as in the proof of Proposition \ref{prop-apriori-estm}, we can get the following estimate
  \begin{align}
&\|u(t)\|^{2}
-\int_0^t \!\left\langle u(s),\,(-D_i\theta^{ik}u+2\beta^k u+2h^k)(s)\,dW^k_s\right\rangle\nonumber
\\
&\leq\|\underline u_0\|^{2}
-(1-\eps)\int_0^t\sum_{k=1}^{d_1}\| L_ku(s)\|^2\,ds 
+C_{\eps}\int_0^t \|u(s)\|^{2}\,ds
\nonumber\\
&\quad + \int_0^t \left(\|h(s)\|^2+ 2\left\langle u(s),\,(L_k'g^k+cu+f)(s)\right\rangle\right)\,ds\quad \text{a.s.},\,\,\forall \eps \in(0,1).
\label{est-l2-3}
\end{align}
 Assume further $c\geq 0$. Put $u_{\lambda}=(u-{\lambda})^+:=\max\{u-{\lambda},0\}$ for ${\lambda}\in [0,\infty)$. If we start from the It\^o formula for the square norm of the positive part of solution (see \cite[Corollary 3.11]{QiuWei-RBSPDE-2013}), in a similar way to the above estimate, we have
\begin{align}
&\|u_{\lambda}(t)\|^{2}
-\int_0^t \!\!\!\langle u_{\lambda}(s),\,(-D_i\theta^{ik}u_{\lambda}+2\beta^k u_{\lambda}+2h^k)(s)\,dW^k_s\rangle\nonumber
\\
&\leq\|u_{\lambda}(0)\|^{2}
-(1-\eps)\int_0^t\sum_{k=1}^{d_1}\| L_ku_{\lambda}(s)\|^2\,ds 
+C_{\eps}\int_0^t \left(\|u_{\lambda}(s)\|^{2}+\langle |u_{\lambda}|,\,\lambda 1_{\{u_{\lambda}>0\}}\rangle (s)\right)\,ds
\nonumber\\
&\quad + \int_0^t \left(\|h(s)1_{\{u_{\lambda}>0\}}\|^2+ 2\left\langle u_{\lambda}(s),\,(L_k'g^k+f)(s)\right\rangle\right)\,ds\quad \text{a.s.},\,\,\forall \eps \in(0,1).
\label{est-l2-uk-1}
\end{align}
where  we note that $u\leq u_{\lambda} + {\lambda}1_{\{u_{\lambda}>0\}}$.
\end{rmk}

Note that we do not assume the H\"ormander-type condition $(\cH)$ in Theorem \ref{thm-BSPDE}. In fact, we may get more regularity properties of solutions of SPDE \eqref{SPDE} under condition $(\cH)$, for which we first recall an estimate on the Lie bracket. 

\begin{lem}\label{lem-lie-bracket}(\cite[Lemma 4.1]{Qiu-2014-Hormander}).
For $\{J,L\}\subset\cup_{l\geq 0} \bV_l$, $m\in\bR$ and $\eps\in[0,1]$, there exists a positive constant $C$ such that almost surely for any $\phi\in H^{m}$ with $J\phi\in H^{m-1+\eps}$ and $L\phi\in H^m$, it holds that
\begin{align*}
\|[J,L]\phi\|_{m-1+\frac{\eps}{2}}
\leq C\left(
\|J\phi\|_{m-1+\eps}+\|L\phi\|_{m}+\|\phi\|_m
\right).
\end{align*}
\end{lem}

The above lemma basically generalizes \cite[Lemma 4.2]{krylov2014hmander-PDE} from the deterministic case when $m=0$ to the stochastic case for any $m\in\bR$. Starting from estimate \eqref{estim-thm} of Theorem \ref{thm-BSPDE} and applying Lemma \ref{lem-lie-bracket} iteratively to elements of $\mathbb{V}_0,\dots,\mathbb{V}_{n_0}$, we have
\begin{cor}\label{cor-grad-est}
Assume the same hypothesis of Theorem \ref{thm-BSPDE}. Let condition $(\cH)$ hold. For the unique solution $u$ of SPDE \eqref{SPDE}, one has further $u\in\cL^2(H^{m+\eta})$ with
$$
E\int_0^T\|u(s)\|_{m+\eta}^2\,ds
\leq
C\left\{
E\|\underline{u}_0\|_m^2 + E\int_0^T\left( \|f(s)\|_m^2+\|g(s)\|_{m}^2+\|h(s)\|_m^2  \right)ds
\right\},
$$
where the constant $C$ depends on $T,m,n_0,\sigma,\theta,b,c$ and $\beta$.
\end{cor}

The estimate on solution of SPDE \eqref{SPDE} for the case $m=0$ in Corollary \ref{cor-grad-est} plays an important role in Section \ref{sec:lp-inf} for the maximum principle of weak solutions. Therefore, for the reader's convenience, we would provide a sketched proof of Lemma \ref{lem-lie-bracket} from which Corollary \ref{cor-grad-est} follows immediately.
\begin{proof}[Proof of Lemma \ref{lem-lie-bracket}]
Assume first $\phi\in H^{m+1}$. Setting $A^n=I^{n-1}[J,L]$, we have $A^n\in\Psi_{n}$ almost surely for each $n\in\bR$. As the adjoint operator of $J$ and $L$, $J^*=-J+\tilde{c}$ and $L^*=-L+\bar{c}$ with $\tilde{c},\bar{c} \in \cL^{\infty}(C^{\infty}_b)$. By Lemma \ref{lem-pdo}, we have
\begin{align*}
&\langle JL\phi,\,I^mA^{m-1+\eps}\phi \rangle\\
&=\langle L\phi,\,(I^mJ^*+[J^*,I^m])A^{m-1+\eps}\phi\rangle\\
&=\langle I^mL\phi,\,(A^{m-1+\eps}J^*+[J^*,A^{m-1+\eps}])\phi  \rangle 
+\langle [I^m,J]L\phi,\,A^{m-1+\eps}\phi\rangle\\
&\leq C\left( \|L\phi\|_m^2+ \|J\phi\|^2_{m-1+\eps}+\|\phi\|_m^2   \right)
\end{align*}
and 
\begin{align*}
&\langle LJ\phi,\,I^mA^{m-1+\eps}\phi \rangle\\
&=\langle J\phi,\,(I^{m-1+\eps}L^*+[L^*,I^{m-1+\eps}])A^m\phi\rangle\\
&=\langle I^{m-1+\eps}J\phi,\,(A^mL^*+[L^*,A^m])\phi\rangle
+\langle I^{m-1+\eps}J\phi,\, I^{-(m-1+\eps)}[L^*,I^{m-1+\eps}]A^m\phi\rangle\\
&\leq
C\left(\|J\phi\|_{m-1+\eps}^2+\|L\phi\|_m^2+\|\phi\|_m^2 \right).
\end{align*}
Hence, 
\begin{align*}
\|[J,L]\phi\|_{m-1+\frac{\eps}{2}}
=\langle [J,L]\phi,\,I^mA^{m-1+\eps}\phi \rangle^{\frac{1}{2}}
\leq C\left(\|J\phi\|_{m-1+\eps}+\|L\phi\|_m+\|\phi\|_m \right).
\end{align*}
Through standard density arguments, one verifies that the above estimate also holds for any $\phi\in H^{m}$ with $J\phi\in H^{m-1+\eps}$ and $L\phi\in H^m$.  
\end{proof}

\subsection{H\"ormander-type theorem for SPDEs}

Inspired by the filtering theory of partially observable diffusion processes,  Krylov \cite{krylov2013hypoellipticity,krylov2013-Hormder-SPDE} has just obtained the H\"ormander-type theorem for  SPDEs, which states the spatial smoothness of solutions. The method therein relies on the generalized It\^o-Wentzell formula and associated results on deterministic PDEs. Next to the above established $L^2$-theory, we intend to derive the following H\"ormander-type theorem for SPDE \eqref{SPDE} under the condition $(\cH)$ with an analytical approach. 

\begin{thm}\label{thm-hormander}
Let assumptions $({\mathcal H} )$ and $(\cA 1)$ hold.
If $f\in\cap_{n\in\bR}\cL^2(H^n)$, $g,h\in\cap_{n\in\bR}\cL^2((H^n)^{d_1})$, and $\underline{u}_0\in L^2(\Omega;H^m)$ for some $m\in\bR$, then for the unique solution $u$ of SPDE \eqref{SPDE} in Theorem \ref{thm-BSPDE}, one has for any $\eps\in(0,T)$,
$$u\in\cap_{n\in\bR} L^2(\Omega;C([\eps,T];H^n)),$$ 
and for any $n\in\bR$,
\begin{align}
&E\sup_{t\in[\eps,T]}\| u(t)\|_{n}^2
+E\int_{\eps}^{T} \| u(t)\|_{n+\eta}^2\,dt
\nonumber\\
&\leq C\left\{E\|\underline{u}_0\|_m^2+
E\int_{0}^{T} \left(\|f(s)\|_{n}^2+ \|g(s)\|_{n}^2+ \|h(s)\|_{n}^2\right) ds\right\},\label{est-hmder}
\end{align}
with the constant $C$ depending on $\eps,n,T,m,n_0,\sigma,\theta,\gamma,b$ and $c$. In particular, the random field $u(t,x)$ is almost surely infinitely differentiable with respect to $x$ on $(0,T]\times\bR^d$ and each derivative is a continuous function on $(0,T]\times\bR^d$.
\end{thm}
\label{sec:proof-main-thm}

\begin{proof}
By Theorem \ref{thm-BSPDE}, SPDE \eqref{SPDE} admits a unique solution $u\in\cS^2(H^m)$ and the random field $\bar{u}(t,x):=tu(t,x)$ is the unique solution of SPDE
\begin{equation}\label{SPDE-1}
  \left\{\begin{array}{l}
  \begin{split}
  d\bar{u}(t,x)=\,&\displaystyle \left[ \frac{1}{2}(L_k^2+M_k^2)\bar{u}
         +{b}^jD_j\bar{u}+c\bar{u}+u+t\left(f+L_k'g^k+M_k'h^k\right)
                \right](t,x)\, dt\\ &\displaystyle
           +\left[tM_k\bar{u}+t\beta^k \bar{u}+th^k\right](t,x)\, dW_{t}^{k}, \quad
                     (t,x)\in Q;\\
    \bar{u}(0,x)=\, &0, \quad x\in\bR^d,
    \end{split}
  \end{array}\right.
\end{equation}
with 
\begin{align*}
&E\sup_{t\in[0,T]}\|\bar u(t)\|_m^2+\sum_{k=1}^{d_1}E\int_{0}^T\|L_k\bar u(t)\|_m^2dt
\nonumber\\
&\leq C\left(T^2+1\right)E\int_{0}^T \left(\|f(s)\|_m^2+\|g(s)\|_m^2+\|h(s)\|_m^2+\|u(s)\|_m^2\right)\,ds.
\end{align*}

Starting from the above estimate, we apply Lemma \ref{lem-lie-bracket} iteratively  to elements of  $\bV_0,\dots,\bV_{n_0}$. Under condition $(\cH)$, there arrives the estimate
\begin{align}\label{grad-est}
\int_0^T\|D\bar{u}\|^2_{m-1+\eta}ds\leq 
C\left(T^2+1\right)E\int_{0}^T \left(\|f(s)\|_m^2+\|g(s)\|_m^2+\|h(s)\|_m^2+\|{u}(s)\|_m^2\right)\,ds.
\end{align}
Fix any $\eps\in(0,T\wedge 1)$ and define $\eps_l=\sum_{i=1}^l\frac{\eps}{2^i}$ for $l\in\bN^+$. By interpolation and Theorem \ref{thm-BSPDE}, we have 
\begin{align*}
&E\sup_{t\in[\eps_1,T]}\| u(t)\|_m^2
+E\int_{\eps_1}^{T}\| u(t)\|_{m+2^{-n_0}}^2 dt
\nonumber\\
&
\leq \frac{C2(T^2+1)}{\eps}E\int_{0}^{T} \left(\|f(s)\|_m^2+\|g(s)\|_m^2+\|h(s)\|_m^2+\|u(s)\|_m^2\right)\,ds.
\end{align*}

Since $f\in\cap_{n\in\bR}\cL^2(H^n)$ and $g,h\in \cap_{n\in\bR}\cL^2((H^n)^{d_1})$, by iteration we obtain for any $j\in\bN^+$,
\begin{align}
&E\sup_{t\in[\eps_j,T]}\| u(t)\|_{m+(j-1)\eta}^2
+E\int_{\eps_j}^{T} \| u(t)\|_{m+j\eta}^2\,dt
\nonumber\\
&\leq \frac{C2^j(T^2+1)}{\eps}
E\int_{\eps_{j-1}}^{T} \!\!\!\left(\|f(s)\|_{m+(j-1)\eta}^2+ \|g(s)\|_{m+(j-1)\eta}^2+ \|h(s)\|_{m+(j-1)\eta}^2+\|u(s)\|_{m+(j-1)\eta}\right)\! ds,\nonumber
\end{align}
which together with estimate \eqref{estim-thm}, implies by iteration that
\begin{align}
&E\sup_{t\in[\eps_j,T]}\| u(t)\|_{m+(j-1)\eta}^2
+E\int_{\eps_j}^{T} \| u(t)\|_{m+j\eta}^2\,dt
\nonumber\\
&\leq C(j)\left\{E\|\underline{u}_0\|_m^2+
E\int_{0}^{T} \left(\|f(s)\|_{m+(j-1)\eta}^2+ \|g(s)\|_{m+(j-1)\eta}^2+ \|h(s)\|_{m+(j-1)\eta}^2\right) ds\right\}.\nonumber
\end{align}

Hence, for any $\eps\in (0,T)$, one has $u\in\cap_{n\in\bR} L^2(\Omega;C([\eps,T];H^n))$ 
and the estimate \eqref{est-hmder} holds. In particular, by Sobolev embedding theorem, $u(t,x)$ is almost surely infinitely differentiable with respect to $x$ and each derivative is a continuous function on $(0,T]\times\bR^d$.
\end{proof}

\begin{rmk}
By Theorem \ref{thm-hormander}, we have the global spatial smoothness of the solution in time interval $(0,T]$. A similar result exists in Krylov's recent work \cite{krylov2013-Hormder-SPDE,krylov2013hypoellipticity}, which states a local spatial smoothness of solution under a H\"ormander-type condition of local type; roughly speaking, as claimed in \cite{krylov2013-Hormder-SPDE}, if a H\"ormander-type condition and all the assumptions on coefficients just hold on a measurable subset $\Omega_0\times(t_1,t_2)\times B \subset \Omega\times[0,\infty)\times\bR^d$ where $\Omega_0\in\sF$, and $B$ is a ball in $\bR^d$, then any solution $u(\omega,t,x)$ satisfying the concerned SPDE on $\Omega_0\times(t_1,t_2)\times B$ admits a version that is, for almost all $(\omega,t)\in \Omega_0\times (t_1,t_2)$, infinitely differentiable with respect to $x$ on $B$. However, the method therein relies on the generalized It\^o-Wentzell formula and associated results on deterministic PDEs, while herein, we use an analytical approach on the basis of our $L^2$-theory and an estimate on the Lie bracket (Lemma \ref{lem-lie-bracket}). In fact, our method has the potential to derive the associated local results, but we would not seek such a generality in the present paper. In addition, we would mention that, to the best of our knowledge, the hypoellipticity for SPDEs was first considered by Chaleyat-Maurel and Michel \cite{chaleyat1984hypoellipticity}, where the coefficients depend on $(t,\omega)$ only through  a substituted Wiener process.

\end{rmk}

\section{$L^p$ estimates for the uniform norm of solutions}
\label{sec:lp-inf}


In this section, let assumptions $(\cA 1)$, $(\cA 2)$ and $(\cH)$ hold. By Theorem \ref{thm-BSPDE}, SPDE \eqref{SPDE} has a unique \textit{weak} solutoin. In this section, we shall prove the $L^p$-estimates for the time-space uniform norm of the weak solution. 
\begin{thm}\label{thm-Lp-inf}
 For the weak solution $u$ of SPDE \eqref{SPDE}, there exists $\theta_0\in (0,1]$ such that for any $p\in(0,\infty)$,
\begin{align*}
E\|u^{\mp}\|^p_{L^{\infty}(Q)} \leq C \left(\Lambda^{\mp}_{\bar p,\infty} + \Lambda^{\mp}_{\frac{p}{\theta_0}}  \right)^p,
\end{align*}
with the constant $C$ depending on $d,p,n_0, T$ and the quantities related to the coefficients $\sigma,\theta,b,c$ and $\beta$.
\end{thm}

An immediate consequence is the following comparison principle. 
\begin{cor}
Suppose that random field $u$ is the weak solution of SPDE \eqref{SPDE}. Let $\tilde u$ be the solution of SPDE \eqref{SPDE} with the initial value $\underline u_0$ and external force $f$ being replaced by $\tilde{\underline u}_0$ and $\tilde f$ respectively. Suppose further that
$$
f\leq \tilde f,\quad \mathbb{P}\otimes dt \otimes dx\text{-a.e. and }\underline{u}_0\leq \tilde{\underline u}_0,\quad \mathbb{P}\otimes dx\text{-a.e.}
$$ 
Then, there holds $u\leq \tilde u$, $\mathbb{P}\otimes dt \otimes dx$-a.e.
\end{cor}

Before proving Theorem \ref{thm-Lp-inf}, we give the following embedding lemma that will be used frequently in what follows.
\begin{lem}\label{lem-embedding}
For $\psi\in L^2(0,T;H^{\eta})\cap C([0,T];L^2)$, one has $\psi\in L^{\frac{2(d+2\eta)}{d}}(Q)$ and 
\begin{align}
\|\psi\|_{L^{\frac{2(d+2\eta)}{d}}(Q)} \leq  \|\psi\|^{\frac{d}{d+2\eta}}_{L^2(0,T;H^{\eta})}  \|\psi\|^{\frac{2\eta}{d+2\eta}}_{C([0,T];L^2)} \label{eq-lem-embedding}
\end{align}
with the positive constant $C$ depending on $d$ and $\eta$.
\end{lem}

\begin{proof}
  By the fractional Gagliard-Nirenberg inequality (see \cite[Corollary 2.3]{Hajaiej-2012-frac-Gargld-Nirenberg} for instance), we have
  \begin{equation*}
    \|\psi(s,\cdot)\|_{L^q}^q
    \leq\  C~
        \|\psi(s,\cdot)\|_{\eta}^{\alpha q}\|\psi(s,\cdot)\|^{q(1-\alpha)},
        \quad \text{a.e.}\ s\in[0,T],
  \end{equation*}
  where $\alpha=d/(d+2\eta)$ and $q=2(d+2\eta)/d$.
  Integrating on $[0,T]$, we obtain
  \begin{equation*}
    \begin{split}
     \int_{Q}|\psi(s,x)|^qdxds
    \leq &\ C~
    \|\psi\|_{\eta}^{2}
        \max_{s\in [0,T]}\|\psi(s,\cdot)\|^{(1-\alpha) q}.
    \end{split}
  \end{equation*}
Therefore,  $\psi\in L^{\frac{2(d+2\eta)}{d}}(Q)$ and there holds \eqref{eq-lem-embedding}.
\end{proof}


For $\lambda>0$ and $z\in\bN_0$, set
\begin{align*}
u_z=(u-\lambda(1-2^{-z}))^+ \quad \textrm{and} \quad U_z=\sup_{t\in[0,T]}\|u_z(t)\|^2+\int_0^T\left( \|u_z(t)\|_{\eta}^2+\sum_{k=1}^{d_1}\|L_ku_z(t)\|^2 \right)dt.
\end{align*}
Obviously, for each $z\in\bN^+$, one has $|D_i u_{z-1}|\geq |D_i u_z| $ for $i=1,\dots,d$, 
\begin{align}
  u_{z-1}\geq  u_z ,\,\, u1_{\{u_z>0\}}=u_z+\lambda(1-2^{-z})1_{\{u_z>0\}} \quad \text{and}\quad 1_{\{u_z>0\}}\leq \left(\frac{2^zu_{z-1}}{\lambda}\right)^{q},\quad\forall \, q>0.\label{relatn-uz}
\end{align}

As an immediate consequence of Lemma \ref{lem-embedding}, there follows 
\begin{cor}
$$\|u_z\|^2_{L^{\frac{2(d+2\eta)}{d}}(Q)} \leq C\, U_z,\quad a.s.
$$
with the constant $C$ depending on $d$ and $\eta$.
\end{cor}

In view of Remark \ref{rmk-ito-forml}, the weak solution $u$ of SPDE \eqref{SPDE} satisfies
\begin{align}
&\|u_z(t)\|^{2}
-\int_0^t \!\!\!\langle u_z(s),\,(-D_i\theta^{ik}u_z+2\beta^k u_z+2h^k)(s)\,dW^k_s\rangle\nonumber
\\
&\leq\|u_z(0)\|^{2}
-(1-\eps)\int_0^t\sum_{k=1}^{d_1}\| L_ku_z(s)\|^2\,ds 
+C_{\eps}\int_0^t \left(\|u_z(s)\|^{2}+\left\langle |u_z|,\,\lambda (1-2^{-z})1_{\{u_z>0\}}\right\rangle (s)\right)\,ds
\nonumber\\
&\quad + \int_0^t \left(\|h(s)1_{\{u_z>0\}}\|^2+ 2\left\langle u_z(s),\,(L_k'g^k+f)(s)\right\rangle\right)\,ds,\quad\text{a.s.,}\,\,\forall \eps \in(0,1).
\label{est-l2-uk-1}
\end{align}
Taking $\eps=1/2$, we have by Gronwall inequality 
\begin{align*}
&\sup_{s\in[0,t]}\|u_z(s)\|^{2}+\int_0^t\sum_{k=1}^{d_1}\| L_ku_z(s)\|^2\,ds
\nonumber
\\
&\leq C\bigg\{
\lambda (1-2^{-z})\int_0^t\langle |u_z|,\,1_{\{u_z>0\}}\rangle (s)\,ds
+\sup_{\tau\in[0,t]}\int_0^{\tau} \!\!\left\langle u_z(s),\,(-D_i\theta^{ik}u_z+2\beta^k u_z+2h^k)(s)\,dW^k_s\right\rangle
\nonumber\\
&\quad 
+ \int_0^t \left(\|h(s)1_{\{u_z>0\}}\|^2+ 2\left|\left\langle u_z(s),\,(L_k'g^k+f^+)(s)\right\rangle\right|\right)\,ds
+\|u_z(0)\|^{2}
\bigg\},\quad\text{a.s.}
\end{align*}
Under condition $(\cH)$, starting from the above estimate and applying Lemma \ref{lem-lie-bracket} iteratively  to elements of  $\bV_0,\dots,\bV_{n_0}$, we get
\begin{align}
&\sup_{s\in[0,t]}\|u_z(s)\|^{2}+\int_0^t\left(\|u_z(s)\|^2_{\eta}+\sum_{k=1}^{d_1}\| L_ku_z(s)\|^2\right)\,ds
\nonumber
\\
&\leq C\bigg\{
\lambda (1-2^{-z})\int_0^t\langle |u_z|,\,1_{\{u_z>0\}}\rangle (s)\,ds
+\sup_{\tau\in[0,t]}\int_0^{\tau} \!\!\left\langle u_z(s),\,(-D_i\theta^{ik}u_z+2\beta^k u_z+2h^k)(s)\,dW^k_s\right\rangle
\nonumber\\
&\quad 
+ \int_0^t \left(\|h(s)1_{\{u_z>0\}}\|^2+ 2\left|\left\langle u_z(s),\,(L_k'g^k+f^+)(s)\right\rangle\right|\right)\,ds
+\|u_z(0)\|^{2}
\bigg\}
,\quad\text{a.s.}\label{eq-ito-uz}
\end{align}

Set 
$$
M_z(t)=\int_0^t\left\langle  u_z(s),\,(-D_i\theta^{ik}u_z+2\beta^k u_z+2h^k)(s)\,dW^k_s  \right\rangle,\quad t\in [0,T].
$$
The proof of Theorem \ref{thm-Lp-inf} is started from the iteration inequality of the following lemma.
\begin{lem}\label{lem-iteration}
Assume $\lambda \geq 2 \Lambda_{\bar p,\infty}^+ >1 $. For the solution of SPDE \eqref{SPDE}, there exists a positive constant $N$ such that for any $z\in\bN^+$,
\begin{align}
U_z\leq \frac{N^z}{\lambda^{2\alpha_0}} \left(  U_{z-1}\right)^{1+\alpha_0} + N\sup_{t\in[0,T]} M_z(t), \quad\text{a.s.}\label{itrn-schem}
\end{align}
where $$
0<\alpha_0:=\frac{(\bar{p}-2\eta)(d+2\eta)}{2\bar p d}-\frac{1}{2}.
$$
\end{lem}

\begin{proof}


We estimate each item involved in relation \eqref{eq-ito-uz}. Since $\bar{p}>d+2\eta$, basic calculations yield that $2<2+4\alpha_0<\frac{2(d+2\eta)}{d}$. Then, it holds that
\begin{align*}
&\lambda (1-2^{-z})\int_0^T\langle |u_z|,\,1_{\{u_z>0\}}\rangle (s)\,ds
\\
&\leq \lambda (1-2^{-z})\int_0^T\left\langle |u_{z-1}|,\,\left(\frac{2^zu_{z-1}}{\lambda}\right)^{1+2\alpha_0} \right\rangle (s)\,ds
\\
&= \frac{(1-2^{-z})2^{(1+2\alpha_0)z}}{\lambda^{2\alpha_0}}  \|u_{z-1}\|_{L^{2+2\alpha_0}(Q)}^{2+2\alpha_0}
\\
&\leq
\frac{(1-2^{-z})2^{(1+2\alpha_0)z}}{\lambda^{2\alpha_0}}  \|u_{z-1}\|_{L^{\frac{2(d+2\eta)}{d}}(Q)}^{(2+2\alpha_0)\eps} \|u_{z-1}\|_{L^{2}(Q)}^{(2+2\alpha_0)(1-\eps)}
\\
&\leq 
\frac{C(1-2^{-z})2^{(1+2\alpha_0)z}}{\lambda^{2\alpha_0}}  \left(U_{k-1}\right)^{{1+\alpha_0}},\quad\text{a.s.}
\end{align*}
where by Lyapunov's inequality, $\eps\in(0,1)$ is chosen to satisfy
$$
\frac{1}{2+2\alpha_0}=\frac{d\eps}{2(d+2\eta)}+\frac{1-\eps}{2}. 
$$

Furthermore, we have
\begin{align*}
&\int_0^T \langle u_z,\,f^+\rangle (s) \,ds\\
&\leq \|u_z\|_{L^{\frac{2(d+2\eta)}{d}}(Q)}  \|f^+\|_{L^{\frac{\bar p (d+2\eta)}{(\bar p + d+2\eta)\eta}}(Q)}
\left(\int_{Q}1_{\{u_z>0\}}\,dxds\right)^{\frac{1}{2}-\frac{\eta}{\bar p}}\\
&\leq 
\|u_z\|_{L^{\frac{2(d+2\eta)}{d}}(Q)}  \|f^+\|_{L^{\frac{\bar p (d+2\eta)}{(\bar p + d+2\eta)\eta}}(Q)}
\left(\int_{Q}
\left|\frac{2^zu_{z-1}}{\lambda}\right|^{\frac{2(d+2\eta)}{d}}
\,dxds\right)^{\frac{1}{2}-\frac{\eta}{\bar p}}
\\
&\leq
\left(\frac{2^z}{\lambda}\right)^{1+2\alpha_0}
\|f^+\|_{L^{\frac{\bar p (d+2\eta)}{(\bar p + d+2\eta)\eta}}(Q)}
\|u_{z-1}\|_{L^{\frac{2(d+2\eta)}{d}}(Q)}^{2+2\alpha_0}  
\\
&\leq C
\left(\frac{2^z}{\lambda}\right)^{1+2\alpha_0}\|f^+\|_{L^{\frac{\bar p (d+2\eta)}{(\bar p + d+2\eta)\eta}}(Q)}
(U_{z-1})^{1+\alpha_0},  \quad\text{a.s.}
\end{align*}
and
\begin{align*}
&\int_0^T|\langle u_z,\,L_k'g^k \rangle (s)|\,ds
\\
&=\int_0^T|\langle L_ku_z,\,g^k \rangle (s)|\,ds
\\
&\leq 
\|L_ku_z\|_{L^2(Q)}\left(\int_Q g^21_{\{u_z>0\}}\,dxds\right)^{\frac{1}{2}}\\
&
\leq 
\|L_ku_z\|_{L^2(Q)} \|g\|_{L^{\frac{\bar p}{\eta}}(Q)}  \left(\int_Q 1_{\{u_z>0\}}\,dxds\right)^{\frac{1}{2}-\frac{\eta}{\bar p}}\\
&\leq
\left(\frac{2^z}{\lambda}\right)^{1+2\alpha_0}  \|g\|_{L^{\frac{\bar p}{\eta}}(Q)}  \|L_ku_{z-1}\|_{L^2(Q)} \|u_{z-1}\|_{L^{\frac{2(d+2\eta)}{d}}(Q)}^{1+2\alpha_0}
\\
&\leq C
\left(\frac{2^z}{\lambda}\right)^{1+2\alpha_0}\|g\|_{L^{\frac{\bar p}{\eta}}(Q)}
(U_{z-1})^{1+\alpha_0},\quad\text{a.s.}
\end{align*}

Let $q=\frac{\bar{p}(d+2\eta)}{\eta(\bar{p}+d+2\eta)}$ and $\tilde q = \frac{q}{q-1}$. There follows $\frac{2(d+2\eta)}{d\tilde{q}}=2+2\alpha_0$ and thus
\begin{align*}
&\int_0^T \|h(s) 1_{\{u_z>0\}} \|^2ds\\
&\leq \|h\|_{L^{2q}(Q)}^{2}\left(  \int_Q 1_{\{u_z>0\}}dxds  \right)^{\frac{1}{\tilde q }}\\
&\leq \|h\|_{L^{2q}(Q)}^{2}\left( \int_Q \left(  \frac{2^zu_{z-1}}{\lambda}  \right)^{\frac{2(d+2\eta)}{d}}dxds  \right)^{\frac{1}{\tilde q }}\\
&= \left(\frac{2^z}{\lambda}\right)^{2+2\alpha_0}  \|h\|_{L^{2q}(Q)}^{2}
  \| u_{z-1}  \|_{L^{\frac{2(d+2\eta)}{d}}(Q)}  ^{2+2\alpha_0}
\\
&\leq \frac{C2^{(2+2\alpha_0)z}}{\lambda^{2+2\alpha_0}}  \|h\|_{L^{2q}(Q)}^{2}
\left(  U_{z-1}   \right)^{1+\alpha_0},\quad\text{a.s.}
\end{align*}


Since $\lambda\geq 2 \Lambda^+_{\bar p,\infty}$, it follows that $u_z(0)\equiv 0$ for any $z\in\bN^+$.
Choosing $N$ to be big enough, we have by relation \eqref{eq-ito-uz},
\begin{align*}
U_z\leq \frac{N^z}{\lambda^{2\alpha_0}} \left(  U_{z-1}\right)^{1+\alpha_0} + N\sup_{t\in[0,T]} M_z(t),\quad \text{a.s.}
\end{align*}
\end{proof}

 Next, let us deal with  the martingale part $M_z(\cdot)$ in the iteration inequality \eqref{itrn-schem}. We shall prove that $M_z(\cdot)$ is comparable with $\left(U_{z-1}\right)^{1+\alpha_0}$, and the techniques are generalized from \cite{hsu-2013-Stoch-DeGiorgi} for the superparabolic cases.

\begin{lem}\label{lem-est-prob}
Let $\lambda\geq \Lambda^+_{\bar p,\infty}$. There exists $N\in(1,\infty)$ such that for any $\kappa,\zeta\in(0,\infty)$, 
$$
\bP\left(\left\{   
\sup_{t\in[0,T]} M_z(t) \geq \kappa\zeta,\,\left(U_{z-1}\right)^{1+\alpha_0} \leq \zeta
\right\}\right)
\leq  \exp\left\{-\frac{\kappa^2\lambda^{4\alpha_0}}{2N^z}\right\},\quad \forall\,z\in\bN^+.
$$
\end{lem}

\begin{proof}
First, we have
\begin{align*}
\langle M_z \rangle_T
&=\sum_{k=1}^{d_1}\int_0^T 
\left|
\left\langle  u_z,\,(-D_i\theta^{ik}u_z+2\beta^k u_z+2h^k)(s)\right\rangle
\right|^2ds\\
&\leq
C\int_0^T \left( \|u_z\|^4 + \|u_z\|^2\|h1_{\{u_z>0\}}\|^2   \right)ds,\quad \text{a.s.}
\end{align*}
with the constant $C$ being independent of $z$. On the other hand, we have
\begin{align*}
&\int_0^T \|u_z\|^2\|h1_{\{u_z>0\}}\|^2ds\\
&\leq \sup_{s\in[0,T]}\|u_z(s)\|^2\int_0^T \|h1_{\{u_z>0\}}\|^2ds\\
&\leq \sup_{s\in[0,T]}\|u_z(s)\|^2 \|h\|_{L^{\frac{\bar p}{\eta}}(Q)}^2 
\left( \int_{Q} 1_{\{u_z>0\}} dxds\right)^{1-\frac{2\eta}{\bar p}}
\\
&\leq 
\|h\|_{L^{\frac{\bar p}{\eta}}(Q)}^2 \sup_{s\in[0,T]}\|u_z(s)\|^2 
\left(\int_{Q}
\left|\frac{2^zu_{z-1}}{\lambda}\right|^{\frac{2(d+2\eta)}{d}}
\,dxds\right)^{1-\frac{2\eta}{\bar p}}
\\
&=\left(\frac{2^z}{\lambda}\right)^{2+4\alpha_0}
\|h\|_{L^{\frac{\bar p}{\eta}}(Q)}^2 \sup_{s\in[0,T]}\|u_z(s)\|^2 
\|u_{z-1}\|^{2+4\alpha_0}_{L^{\frac{2(d+2\eta)}{d}}(Q)}
\\
&
\leq C
\left(\frac{2^z}{\lambda}\right)^{2+4\alpha_0}
\|h\|_{L^{\frac{\bar p}{\eta}}(Q)}^2 
\left(  U_{z-1}   \right)^{2+2\alpha_0},\quad \text{a.s.}
\end{align*}
and 
\begin{align*}
&\int_0^T \|u_z(s)\|^4ds
\\
&\leq
 \sup_{s\in[0,T]}\|u_z(s)\|^2\int_0^T \|u_z\|^2ds\\
&\leq 
 \sup_{s\in[0,T]}\|u_{z}(s)\|^2 
\int_{Q}\left|u_z\right|^2
\left|\frac{2^zu_{z-1}}{\lambda}\right|^{4\alpha_0}
\,dxds
\\
&\leq
\left(\frac{2^z}{\lambda}\right)^{4\alpha_0}
 \sup_{s\in[0,T]}\|u_z(s)\|^2 
\|u_{z-1}\|^{2+4\alpha_0}_{2+4\alpha_0}
\\
&\leq  C
\left(\frac{2^z}{\lambda}\right)^{4\alpha_0}
\left(  U_{z-1}   \right)^{2+2\alpha_0},\quad\text{a.s.}
\end{align*}
Therefore, there exists $N\in (1,\infty)$ such that for any $z\in\bN^+$,
\begin{align}
\langle M_z \rangle_T \leq C \left\{
\left(\frac{2^z}{\lambda}\right)^{4\alpha_0}
+\left(\frac{2^z}{\lambda}\right)^{2+4\alpha_0} \|h\|_{L^{\frac{\bar p}{\eta}}(Q)}^2
\right\}
\left(  U_{z-1}   \right)^{2+2\alpha_0}
\leq \frac{N^z}{\lambda^{4\alpha_0}} \left(  U_{z-1}   \right)^{2+2\alpha_0} ,\quad \text{a.s.}\label{eq-martg}
\end{align}
with the constant $C$ being independent of $z$.

In view of relation \eqref{eq-martg}, $\left(U_{z-1}\right)^{1+\alpha_0} \leq \zeta$ implies that $\langle M_z \rangle_T \leq \gamma := \frac{N^z\zeta^2}{\lambda^{4\alpha_0}}$. Note that there exists a Brownian motion $B$ such that $M_t=B_{\langle M \rangle_t}$. Hence,
\begin{align*}
\bP\left(\left\{   
\sup_{t\in[0,T]} M_z(t) \geq \kappa\zeta,\,\left(U_{z-1}\right)^{1+\alpha_0}\leq \zeta
\right\}\right)
&\leq \bP\left(\left\{   
\sup_{t\in[0,T]} M_z(t) \geq \kappa\zeta,\,\langle M_z \rangle_T \leq \gamma
\right\}\right)
\\
&\leq \bP\left(\left\{   
\sup_{t\in[0,\gamma]} B_t \geq \kappa\zeta
\right\}\right)\\
\text{(by the reflection principle) }
&=2 \bP\left(\left\{   
 B_{\gamma} \geq \kappa\zeta
\right\}\right)
\\
&\leq  \exp\left\{-\frac{\kappa^2\zeta^2}{2\gamma}\right\}=\exp\left\{-\frac{\kappa^2\lambda^{4\alpha_0}}{2N^z}\right\},
\end{align*}
which completes the proof. 
\end{proof}

Combining the iteration inequality \eqref{itrn-schem} and the estimate on martingale part $M_z(\cdot)$, we shall estimate the tail probability of $\|u^{+}\|_{L^{\infty}(Q)}$.

\begin{prop}\label{lem-uplus}
There exist $\theta_0\in(0,1)$ and $\lambda_0\in (1,\infty)$ such that for any $\lambda\geq \lambda_0 $,
\begin{align}
\bP\left(\left\{
\|u^+\|_{L^{\infty}(Q)} >\lambda, \,
U_0 \leq \lambda^{2\theta_0}
\right\}\right)
\leq 2\exp\left\{-\lambda^{2\alpha_0} \right\}. \label{est-lem-uplus}
\end{align}
\end{prop}

\begin{proof}
For $z\in\bN_0$, set 
$$
A_z=\left\{ 
U_z\leq \frac{\lambda^{2\theta_0}}{\nu^z} 
\right\},
$$
with the parameter $\nu>1$ waiting to be determined later. 
Observe that 
\begin{align*}
\left\{
\|u^+\|_{L^{\infty}(Q)} >\lambda, \,
U_0 \leq \lambda^{2\theta_0}
\right\}
\subset 
\cup_{z\in\bN_0} \left(A_z\right)^c \cap A_0
\subset
\cup_{z\in \bN^+} \left(A_z\right)^c\cap A_{z-1}
\end{align*}
which implies that 
\begin{align}
\bP\left(\left\{
\|u^+\|_{L^{\infty}(Q)} >\lambda, \,
U_0 \leq \lambda^{2\theta_0}
\right\}\right)
\leq \sum_{z\in \bN^+} 
\bP \left(\left(A_z\right)^c\cap A_{z-1}\right). \label{eq-lem-prob}
\end{align}

In view of Lemma \ref{lem-iteration}, the event in $\left(A_z\right)^c\cap A_{z-1}$ implies that 
\begin{align*}
\sup_{t\in[0,T]} M_z(t)
&\geq \frac{\lambda^{2\theta_0}}{N\nu^z}-\frac{N^{z-1}}{\lambda^{2\alpha_0-2\theta_0(1+\alpha_0)}\nu^{(z-1)(1+{\alpha_0})}}\\
&=\frac{\lambda^{2\theta_0(1+\alpha_0)}}{\nu^{(z-1)(1+\alpha_0)}}\left[
\frac{\nu^{\alpha_0z-1-\alpha_0}}{N\lambda^{2\alpha_0\theta_0}}
-
\frac{N^{z-1}}{\lambda^{2\alpha_0}}
\right].
\end{align*}
Put
$$
\zeta_z= \frac{\lambda^{2\theta_0(1+\alpha_0)}}{\nu^{(z-1)(1+\alpha_0)}}
\quad\text{and}\quad
\kappa_z=
\frac{\nu^{\alpha_0z-1-\alpha_0}}{N\lambda^{2\alpha_0\theta_0}}
-
\frac{N^{z-1}}{\lambda^{2\alpha_0}},
$$
and take 
$$
\theta_0=\frac{1}{4} \quad \text{and}\quad 
\nu=(2N+1)^{\frac{1}{\alpha_0}}. 
$$
There exists $\lambda_0\in (1,\infty)$ such that for any $\lambda \geq \lambda_0$, one has 
$$\kappa_z\geq \frac{(2N+1)^{z}}{\lambda^{\alpha_0}},\quad \forall \,z\in\bN^+.$$

By Lemma \ref{lem-est-prob}, it follows that for any $z\in\bN^+$,
\begin{align*}
\bP \left(\left(A_z\right)^c\cap A_{z-1}\right)
&\leq 
\bP\left(
\left\{
\sup_{t\in[0,T]} M_z(t) \geq \kappa_z\zeta_z,\,\left(U_{z-1}\right)^{1+\alpha_0} \leq \zeta_z
\right\}
\right)\\
&\leq 
\exp\left\{-\frac{\kappa_z^2\lambda^{4\alpha_0}}{2N^z}\right\}
\leq
\exp\left\{-\frac{(2N+1)^{2z}\lambda^{2\alpha_0}}{2N^z}\right\}
\\
&\leq
\exp\left\{ -2^z\lambda^{2\alpha_0} \right\}
\leq \exp\left\{ -z\lambda^{2\alpha_0} \right\},
\end{align*}
which together with relation \eqref{eq-lem-prob} implies estimate \eqref{est-lem-uplus}.
\end{proof}

Finally, equipped with the above estimate on the tail probability, we are now at a position to prove the $L^p$-estimates for the time-space uniform norm of weak solutions.
\begin{proof}[Proof of Theorem \ref{thm-Lp-inf}]
Taking $z=0$ in relation \eqref{eq-ito-uz} and applying H\"older inequality, we have for $0\leq \tau \leq T$,
\begin{align}
&\sup_{t\in[0,\tau]}\|u^+(t)\|^{2}+\int_0^{\tau}\left(\|u^+(s)\|^2_{\eta}+\sum_{k=1}^{d_1}\| L_ku^+(s)\|^2\right)\,ds
\nonumber
\\
&\leq C\bigg\{
\sup_{t\in[0,{\tau}]}\int_0^t \!\!\left\langle u^+(s),\,(-D_i\theta^{ik}u^++2\beta^k u^++2h^k)(s)\,dW^k_s\right\rangle
\nonumber\\
&\quad 
+ \int_0^{\tau} \left(\|h(s)1_{\{u>0\}}\|^2+ 2\left|\left\langle u^+(s),\,(L_k'g^k+f^+)(s)\right\rangle\right|\right)\,ds
+\|\underline{u}^+_0 \|^{2}
\bigg\}\nonumber\\
&
\leq C\bigg\{
\sup_{t\in[0,{\tau}]}\int_0^t \!\!\left\langle u^+(s),\,(-D_i\theta^{ik}u^++2\beta^k u^++2h^k)(s)\,dW^k_s\right\rangle
\nonumber\\
&\quad 
+ \int_0^{\tau} \left(\|h(s)1_{\{u>0\}}\|^2+ \|g(s)1_{\{u>0\}}\|^2 +\|f^+(s)1_{\{u>0\}}\|^2\right)\,ds
+\|\underline{u}^+_0 \|^{2}
\bigg\}\nonumber\\
&+\frac{1}{2} \int_0^{\tau} \left(
\|u^+\|^2+\sum_{k=1}^{d_1}\|L_ku^+(s)\|^2
\right)ds,\quad\text{a.s.,}\nonumber
\end{align}
which implies that
\begin{align}
&\sup_{t\in[0,\tau]}\|u^+(t)\|^{2}+\int_0^{\tau}\left(\|u^+(s)\|^2_{\eta}+\sum_{k=1}^{d_1}\| L_ku^+(s)\|^2\right)\,ds
\nonumber
\\
&\leq C\left\{\sup_{t\in[0,{\tau}]} \tilde{M}_t + \|(f^+,g,h)1_{\{u>0\}}\|_{L^2([0,{\tau}]\times\bR^d)}^2 + \|\underline{u}^+_0 \|^2 \right\},\quad\text{a.s.,}\label{eq-U0-l2}
\end{align}
with
$$
\tilde M_t:=\int_0^t \!\!\left\langle u^+(s),\,(-D_i\theta^{ik}u^++2\beta^k u^++2h^k)(s)\,dW^k_s\right\rangle,\quad t\in[0,T].
$$

Observe that for any $t\in[0,T]$ and $q> 0$,
\begin{align*}
\langle \tilde M\rangle_t^{\frac{q}{2}} 
&=\left(\sum_{k=1}^{d_1}\int_0^t
\left|
\left\langle  u^+(s),\,(-D_i\theta^{ik}u^++2\beta^k u^++2h^k)(s)\right\rangle
\right|^2ds\right)^{\frac{q}{2}}\\
&\leq
C\left(\int_0^t \left( \|u^+\|^4 + \|u^+\|^2\|h1_{\{u>0\}}\|^2   \right)(s)\,ds\right)^{\frac{q}{2}}
\\
&\leq \left(\eps+C\tau^{\frac{q}{2}}\right) \sup_{s\in[0,t]}\|u^+(s)\|^{2q} +  C_{\eps}  \left(\int_0^t\|h1_{\{u>0\}}\|^2 ds\right)^{q}   .
\end{align*}
Take 
$$\eps=\frac{1}{4} \quad \text{and}\quad   
\tau=T \wedge \left(\frac{1}{4C}\right)^{\frac{2}{q}}.
$$
By relation \eqref{eq-U0-l2} and the Burkholder-Davis-Gundy inequality, we have for $q>0$,
\begin{align*}
&E\sup_{t\in[0,\tau]}\|u^+(t)\|^{2q}+E\left[\int_0^{\tau}\left(\|u^+(s)\|^2_{\eta}+\sum_{k=1}^{d_1}\| L_ku^+(s)\|^2\right)\,ds\right]^q
\nonumber
\\
&\leq 
\frac{1}{2} E\sup_{s\in[0,\tau]}\|u^+(s)\|^{2q} +  CE \left[
\left( \|(f^+,g,h)1_{\{u>0\}}\|_{L^2([0,{\tau}]\times\bR^d)}^2 + \|\underline{u}^+_0 \|^2 \right)^{q}\right].
\end{align*}
Starting from the interval $[0,\tau]$,  within $\left\lceil\frac{T}{\tau} \right\rceil$ steps we arrive at
\begin{align*}
E\left(U_0\right)^q
\leq C E\left[\|(f^+,g,h)1_{\{u>0\}}\|_{L^2(Q)}^{2q} + \|\underline{u}^+_0 \|^{2q} \right].
\end{align*}

Taking $q=\frac{p}{2\theta_0}$ in the above inequality, we have by Proposition \ref{lem-uplus},
\begin{align*}
&E\|u^+\|^p_{L^{\infty}(Q)}\\
&=p\int_0^{\infty} \bP\left( \left\{ \|u^+\|_{L^{\infty}(Q)} > \lambda \right\}\right)\lambda^{p-1}\,d\lambda\\
&\leq
\lambda_0^p
+\int_{\lambda_0}^{\infty} \bP\left(\left\{ U_0 > \lambda^{2\theta_0}  \right\}\right)\lambda^{p-1} \,d\lambda
+\int_{\lambda_0}^{\infty} \bP\left(\left\{
\|u^+\|_{L^{\infty}(Q)} > \lambda,\, U_0 \leq \lambda^{2\theta_0}\right\}  \right)\lambda^{p-1} \,d\lambda
\\
&\leq
\lambda_0^p
+\frac{1}{2\theta_0} E\left|U_0\right|^{\frac{p}{2\theta_0}}
+\int_{\lambda_0}^{\infty}  2 \exp\left\{-\lambda^{2\alpha_0 }\right\} \lambda^{p-1} \,d\lambda
\\
&<\infty.
\end{align*}

Hence, in view of Lemmas \ref{lem-iteration} and \ref{lem-uplus}, we have by scaling 
\begin{align*}
E\|u^+\|^p_{L^{\infty}(Q)} \leq C \left(\Lambda^+_{\bar p,\infty } + \Lambda^+_{\frac{p}{\theta_0}}  \right)^p,
\end{align*}
with the constant $C$ depending on $d,p,n_0, T$ and the quantities related to the coefficients $\sigma,\theta,b,c$ and $\beta$. The estimate on $u^-$ follows in a similar way. We complete the proof.
\end{proof}

\begin{rmk}
Theorem \ref{thm-Lp-inf} addresses the $L^p$ ($p>0$) estimates for the time-space uniform norm of weak solutions for possibly \textit{degenerate} SPDE \eqref{SPDE} in the whole space. It seems to be new, even for the \textit{super-parabolic} case (that is $n_0=0$ in $(\cH)$), as the existing results on such kind of estimates for weak solutions of \textit{super-parabolic} SPDEs are restricted in bounded domains (see \cite{D-bounddeness-SPDE-2013,DenisMatoussi2009,DenisMatoussiStoica2005}) with $p\in[2,\infty)$. In fact, our method of De Giorgi iteration in this section is applicable to the local maximum principle for weak solutions of SPDEs in either bounded or unbounded domains, by using the techniques of cut-off functions (see \cite{QiuTangMPBSPDE11} for instance). On the other hand, in Theorem \ref{thm-Lp-inf} as well as in assertion (i) of Theorem \ref{thm-main},  we assume $(\cA1)$ which requires the spatial smoothness of coefficients $\sigma,\theta,b,c$ and $\beta$; in fact, such assumption is made for the sake of simplicity and it can be relaxed in a standard way due to the properties of multipliers in (ii) of Lemma \ref{lem-pdo}. However, we would postpone such generalizations in domains with relaxed assumption $(\cA1)$ to a future work. 
\end{rmk}


\bibliographystyle{siam}
%

\begin{thebibliography}{10}

\bibitem{chaleyat1984hypoellipticity}
{\sc M.~Chaleyat-Maurel and D.~Michel}, {\em Hypoellipticity theorems and
  conditional laws}, Probab. Theory Relat. Fields, 65 (1984), pp.~573--597.

\bibitem{Deg-SPDE-3-2014}
{\sc K.~Dareiotis}, {\em A note on degenerate stochastic integro-differential
  equations}, arXiv:1406.5649,  (2014).

\bibitem{D-bounddeness-SPDE-2013}
{\sc K.~Dareiotis and M.~Gerencs{\'e}r}, {\em On the boundedness of solutions
  of {SPDEs}}, Stochastic Partial Differential Equations: Analysis and
  Computations,  (2013), pp.~1--19.

\bibitem{dareiotis2017local}
{\sc K.~Dareiotis and M.~Gerencs{\'e}r},  {\em Local $L^{\infty}$-estimates,
  weak {Harnack} inequality, and stochastic continuity of solutions of spdes}, J.
  Diff. Eq., 262 (2017), pp.~615--632.

\bibitem{dareiotis2017supremum}
{\sc K.~Dareiotis and B.~Gess}, {\em Supremum estimates for degenerate,
  quasilinear stochastic partial differential equations}, arXiv preprint
  arXiv:1712.06655,  (2017).

\bibitem{DenisMatoussi2009}
{\sc L.~Denis and A.~Matoussi}, {\em Maximum principle and comparison theorem
  for quasi-linear stochastic {PDE}s}, Electron. J. Probab., 14 (2009),
  pp.~500--530.

\bibitem{DenisMatoussiStoica2005}
{\sc L.~Denis, A.~Matoussi, and L.~Stoica}, {\em $\textrm{L}^p$ estimates for
  the uniform norm of solutions of quasilinear $\textrm{SPDE}$'s}, Probab.
  Theory Relat. Fields, 133 (2005), pp.~437--463.

\bibitem{denis-2014-MP}
{\sc L.~Denis, A.~Matoussi, and J.~Zhang}, {\em Maximum principle for
  quasilinear stochastic {PDEs} with obstacle}, Electron. J. Probab, 19 (2014),
  pp.~1--32.

\bibitem{G-2019}
{\sc M.~Gerencs$\acute{e}$r}, {\em Boundary regularity of stochastic pdes},
  Ann. Probab., 47 (2019), pp.~804--834.

\bibitem{Deg-SPDE-2014}
{\sc M.~Gerencs{\'e}r, I.~Gy{\"o}ngy, and N.~Krylov}, {\em On the solvability
  of degenerate stochastic partial differential equations in {Sobolev} spaces},
  Stochastic Partial Differential Equations: Analysis and Computations,
  (2014), pp.~1--32.

\bibitem{Hajaiej-2012-frac-Gargld-Nirenberg}
{\sc H.~Hajaiej, X.~Yu, and Z.~Zhai}, {\em Fractional {Gagliardo--Nirenberg}
  and {Hardy} inequalities under {Lorentz} norms}, J. Math. Anal. Appl., 396
  (2012), pp.~569--577.

\bibitem{Hormander1983analysis}
{\sc L.~H{\"o}rmander}, {\em The Analysis of Linear Partial Differential
  Operators {III}}, vol.~257, Springer, 1983.

\bibitem{hsu-2013-Stoch-DeGiorgi}
{\sc E.~P. Hsu, Y.~Wang, and Z.~Wang}, {\em Stochastic de giorgi iteration and
  regularity of stochastic partial differential equations}, Ann. Probab., 45
  (2017), pp.~2855--2866.

\bibitem{Kryl96}
{\sc N.~V. Krylov}, {\em On $\textrm{L}_p$-theory of stochastic partial
  differential equations}, \textrm{SIAM} J. Math. Anal., 27 (1996),
  pp.~313--340.

\bibitem{krylov2013-Hormder-SPDE}
{\sc N.~V. Krylov}, {\em H\"ormander's theorem for stochastic partial
  differential equations}, arXiv:1309.5543,  (2013).

\bibitem{krylov2013hypoellipticity}
{\sc N.~V. Krylov}, {\em Hypoellipticity for
  filtering problems of partially observable diffusion processes}, Probab.
  Theory Relat. Fields,  (2013), pp.~1--32.

\bibitem{krylov2014hmander-PDE}
{\sc N.~V. Krylov}, {\em H\"ormander's
  theorem for parabolic equations with coefficients measurable in the time
  variable}, SIAM J. Math. Anal., 46 (2014), pp.~854--870.

\bibitem{krylov-1977-Deg-L2}
{\sc N.~V. Krylov and B.~L. Rozovskii}, {\em On the cauchy problem for linear
  stochastic partial differential equations}, Izvestiya: Mathematics, 11
  (1977), pp.~1267--1284.

\bibitem{Krylov_Rozovskii81}
{\sc N.~V. Krylov and B.~L. Rozovskii}, {\em Stochastic evolution equations},
  J. Sov. Math., 16 (1981), pp.~1233--1277.

\bibitem{krylov_1986charcteristic}
{\sc N.~V. Krylov and B.~L. Rozovskii}, {\em Characteristics of degenerating
  second-order parabolic it{\^o} equations}, J. Math. Science, 32 (1986),
  pp.~336--348.

\bibitem{Deg-SPDE-2-2014}
{\sc J.-M. Leahy and R.~Mikulevi{\v{c}}ius}, {\em On degenerate linear
  stochastic evolution equations driven by jump processes}, Stoch. Proc.Appl.,
  125 (2015), pp.~3748--3784.

\bibitem{Qiu-2014-Hormander}
{\sc J.~Qiu}, {\em H{\"o}rmander-type theorem for it{\^o} processes and related
  backward spdes}, Bernoulli, 24 (2018), pp.~956--970.

\bibitem{QiuTangMPBSPDE11}
{\sc J.~Qiu and S.~Tang}, {\em Maximum principles for backward stochastic
  partial differential equations}, J. Funct. Anal., 262 (2012), pp.~2436--2480.

\bibitem{QiuWei-RBSPDE-2013}
{\sc J.~Qiu and W.~Wei}, {\em On the quasi-linear reflected backward stochastic
  partial differential equations}, J. Funct. Anal., 267 (2014), pp.~3598--3656.

\bibitem{wang2018probabilistic}
{\sc Z.~Wang}, {\em A probabilistic harnack inequality and strict positivity of
  stochastic partial differential equations}, Probab. Theory Relat. Fields, 171
  (2018), pp.~653--684.

\end{thebibliography}

\end{document}